\documentclass[12pt]{article}

\usepackage{amsmath, amssymb, amsthm}
\usepackage{geometry}
\usepackage{hyperref}
\usepackage{enumitem}
\usepackage{xcolor}
\usepackage{tikz}
\usepackage{graphicx}

\DeclareMathOperator{\End}{end}
\DeclareMathOperator{\dist}{dist}

\newtheorem{theorem}{Theorem}[section]
\newtheorem{lemma}[theorem]{Lemma}
\newtheorem{proposition}[theorem]{Proposition}
\newtheorem{corollary}[theorem]{Corollary}

\newtheorem{claim}{Claim}[theorem]

\theoremstyle{definition}
\newtheorem{definition}[theorem]{Definition}

\newtheorem{remark}[theorem]{Remark}

\title{On graphs without cycles of length $0$ modulo $3$ or $4$ modulo $6$}
\author{Takahiro Ueoro \thanks{Yokohama National University, Japan. \texttt{ueoro-takahiro-wc@ynu.jp}. }}
\date{\today}

\begin{document}

\maketitle

%abst
\begin{abstract}
We study graphs containing no cycle whose length is divisible by $3$ or congruent to $4$ modulo $6$. 
We prove that every such $n$-vertex graph $G$, where $n \ge 2$, satisfies $e(G) \le (11/8)n-7/4$. 
Moreover, equality holds if and only if $n=8k+2$ for some nonnegative
integer $k$ and $G$ is isomorphic to the explicitly constructed graph $H_k$. 
We also construct, for every $n\geq2$, an $n$-vertex graph with $\left\lfloor (11/8)n-7/4 \right\rfloor$ 
edges satisfying the same cycle restriction. 
Consequently, this is the exact maximum number of edges for every $n \ge 2$. 
\end{abstract}

\noindent\textbf{Keywords:}
extremal graph theory, cycles modulo $k$, forbidden cycles,
planar graphs, extremal number

%intro
\section{Introduction}
Extremal problems concerning graphs with forbidden cycles of prescribed lengths are fundamental topics in extremal graph theory and have been widely studied. 
One such problem is to determine the maximum number of edges in a graph containing no cycle whose length belongs to a prescribed residue class modulo $k$, and this problem has attracted considerable attention in recent years. 

Let $\mathcal{F}$ be a set of graphs. 
A graph is called \emph{$\mathcal{F}$-free} if it contains no member of $\mathcal{F}$ as a subgraph. 
We denote by $ex(n, \mathcal{F})$ the maximum number of edges in an $n$-vertex \emph{$\mathcal{F}$-free} graph. 
Let $k$ be a positive integer, and let $0\leq \ell_1<\ell_2<\cdots<\ell_m<k$. 
A cycle whose length is congruent modulo $k$ to one of $\ell_1,\ell_2,\ldots,\ell_m$ is called an \emph{$(\ell_1,\ell_2,\ldots,\ell_m \bmod k)$-cycle}. 
We denote by $\mathcal{C}_{\ell_1,\ell_2,\ldots,\ell_m \bmod k}$ the family of all $(\ell_1,\ell_2,\ldots,\ell_m \bmod k)$-cycles. 
%In particular, when $m=1$, we write $\mathcal{C}_{\ell_1 \bmod k}$. 
For example, we have $ex(n, \mathcal{C}_{0 \bmod 1}) = n - 1$ since a graph contains no $(0 \bmod 1)$-cycles if and only if it is a forest. 
 
When $(k, \ell) \equiv (0,1) \pmod{2}$, every bipartite graph, and in particular the complete bipartite graph $K_{\lfloor n/2 \rfloor,\, \lceil n/2 \rceil}$, contains no $(\ell \bmod k)$-cycles.  
Therefore,
$ex(n, \mathcal{C}_{\ell \bmod k}) \ge \left\lfloor n/2 \right\rfloor \left\lceil n/2 \right\rceil = \left\lfloor n^2/4 \right\rfloor$. 
Furthermore, by the theorem of Simonovits~\cite{Simonovits1974},  
when $n$ is sufficiently large, we have
\[
ex(n, \mathcal{C}_{\ell \bmod k}) = \left\lfloor \frac{n^{2}}{4} \right\rfloor.
\]
Thus this case is essentially settled. 

Throughout this section, we assume that $(k,\ell) \not\equiv (0,1) \pmod{2}$.  
Burr and Erd\H{o}s~\cite{Erdos1976} conjectured that $ex(n, \mathcal{C}_{\ell \bmod k}) = O(n)$,  
and this was proved by Bollob\'{a}s~\cite{Bollobas1977}.
Let $c_{\ell,k}$ denote the supremum of the set 
$\left\{ ex(n, \mathcal{C}_{\ell \bmod k})/n \,\middle|\, n \text{ is a positive integer} \right\}$.
Various upper bounds on $c_{\ell,k}$ have been improved in various ways, and several results are now known. 
For example, Sudakov and Verstra\"{e}te~\cite{Sudakov2017} showed that for $k > \ell \ge 3$,
\[\frac{ex(k, C_{\ell})}{k} \le c_{\ell,k} \le \frac{96 \cdot ex(k, C_{\ell})}{k}. \]
Therefore, $c_{\ell,k}$ and $ex(k, C_{\ell})/k$ differ only by a constant factor.

There are not many pairs $(k,\ell)$ for which the value of $c_{\ell,k}$ is known.
It is well known that $c_{0,1} = 1$ and $c_{0,2} = 3/2$. 
Chen and Saito~\cite{Chen1994} proved that every graph with minimum degree at least $3$ contains a $(0 \bmod 3)$-cycle. This result implies that $c_{0,3} = 2$. 
Bai, Li, Pan and Zhang~\cite{Li2025} proved that if $G$ is an $n$-vertex $\mathcal{C}_{1 \bmod 3}$-free graph, then $e(G) \le 5(n-1)/3$. Moreover, equality holds if and only if $9 \mid (n-1)$ and every block of $G$ is isomorphic to the Petersen graph. This result implies that $c_{1,3} = 5/3$. 
Gy\H{o}ri, Li, Salia, Tompkins, Varga and Zhu~\cite{Gyori2026} proved that if $G$ is an $n$-vertex $\mathcal{C}_{0 \bmod 4}$-free graph, then $e(G) \le 19(n-1)/12$. Moreover, they constructed infinitely many extremal examples. This result implies that $c_{0,4} = 19/12$. 
Cai and Shreve~\cite{Cai2001} proved that every $2$-connected graph on at least six vertices with minimum degree at least $3$ contains a $(2 \bmod 4)$-cycle. By a standard induction on the block decomposition, their result implies that $e(G) \le 5(n-1)/2$ for every $n$-vertex $\mathcal{C}_{2 \bmod 4}$-free graph $G$, with equality if and only if $G$ is connected and every block of $G$ is isomorphic to $K_5$. In particular, $c_{2,4}=5/2$. 
More recently, Gao, Li, Ma and Xie~\cite{Gao2024} proved the stronger result that every $n$-vertex graph $G$ with at least $5(n-1)/2$ edges contains two cycles of consecutive even lengths unless $G$ is connected and every block of $G$ is isomorphic to $K_5$.

Besides the results mentioned above, the values of $c_{\ell,k}$ have been determined for infinitely many pairs $(k, \ell)$. Gao, Huo, Liu and Ma~\cite{Gao2022} proved that, when $k \ge 3$, every graph with minimum degree at least $k+1$ contains an $(\ell \bmod k)$-cycle. This result implies that $c_{\ell,k} \le k$ for $k \ge 3$. Moreover, since $K_{k,n-k}$ contains no $(2 \bmod k)$-cycle when $k$ is odd, it implies that $c_{2,k} = k$ for every odd integer $k \ge 3$. 
Bai, Grzesik, Li and Prorok~\cite{Bai2025} proved that, when $k \ge 4$, every $2$-connected $n$-vertex graph $G$ with minimum degree at least $k$ contains an $(\ell \bmod k)$-cycle unless $G$ is isomorphic to $K_{k+1}$ or $K_{k,n-k}$. 
Using this result, they also showed that $c_{\ell,k} \le k-1$ whenever $\ell \neq 2$. Furthermore, since $K_{k-1,n-k+1}$ contains no $(0 \bmod k)$-cycle when $k$ is odd, it implies that $c_{0,k}=k-1$ for every odd integer $k \ge 5$. 

Summarizing the results mentioned above, the values of $c_{\ell,k}$ have been determined when $k \le 4$, and when $k \ge 5$ is odd with $\ell \in \{0,2\}$. 

The corresponding extremal problem has also been studied within the class of $2$-connected graphs. 
Chu, Park and Ryu~\cite{Chu2025} proved that every $2$-connected $n$-vertex graph containing no $(0 \bmod 4)$-cycle has at most $ \lfloor (3n-1)/2 \rfloor $ edges, and constructed infinitely many graphs attaining this bound. 
More recently, Bai, Chu, Li, Park and Ryu~\cite{Bai2026} determined the exact maximum number of edges and characterized all extremal graphs in the $2$-connected setting for each of the two cases of forbidding $(1 \bmod 3)$-cycles and $(2 \bmod 4)$-cycles. 
These results show that imposing $2$-connectivity can change both the maximum number of edges and the structure of the extremal graphs. 
 
Most previous work has focused on forbidding a single congruence class of cycle lengths. 
It is natural to consider several classes simultaneously, since their interaction may lead to stronger upper bounds on the number of edges and new extremal structures. 

In this paper, we study graphs containing no cycle whose length is divisible by $3$ or congruent to $4$ modulo $6$; equivalently, graphs containing no $(0,3,4 \bmod 6)$-cycle. 
We determine the exact maximum number of edges in such graphs and characterize all extremal graphs. 
To the best of our knowledge, this is the first systematic study of the extremal consequences of simultaneously forbidding multiple congruence classes of cycle lengths. 
We hope that this work will stimulate further research in this direction. 

Before stating the main theorem, we define the graph $H_k$. 

\begin{definition}
Let $A$ be the graph with two designated vertices $x$ and $y$ shown in the upper right of Figure~\ref{fig_Hk}. 
For a positive integer $k$, take $k$ copies $A_1,\ldots,A_k$ of $A$, and identify all vertices corresponding to $x$ into a single vertex denoted by $x_k$. 
Similarly, identify all vertices corresponding to $y$ into a single vertex denoted by $y_k$. 
The graph obtained in this way, together with the additional edge $x_k y_k$, is denoted by $H_k$. 
We also define $H_0 = K_2$. 
\end{definition}

%%%%%%%%%%%%
\begin{figure}[htbp]
\centering
\includegraphics[width = 10cm]{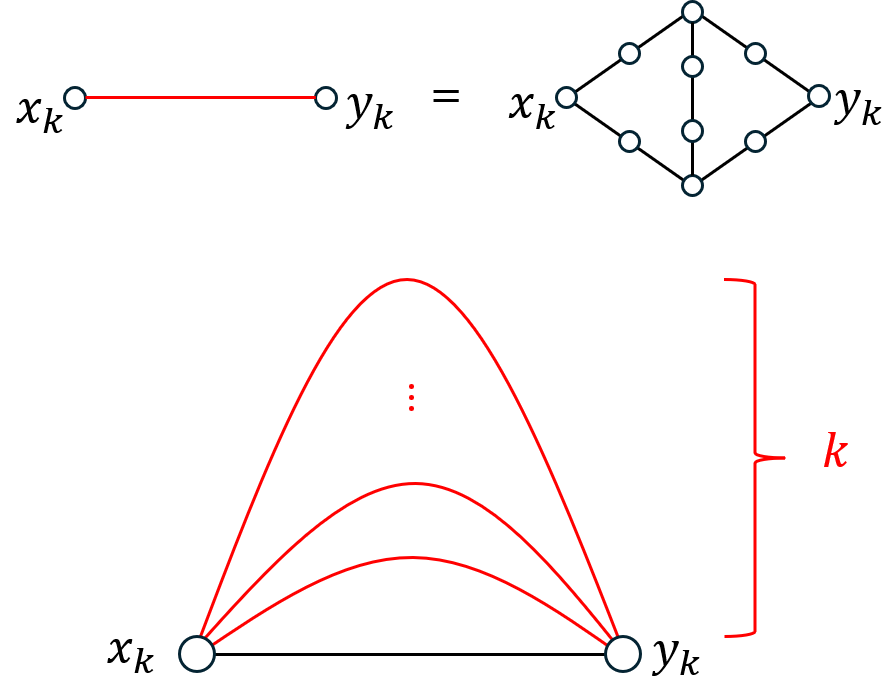}
\caption{The graph $A$ (upper right) and the graph $H_k$ (bottom). }
\label{fig_Hk}
\end{figure}
%%%%%%%%%%%%

\begin{remark} \label{rmk:Hk}
We have $|V(H_k)| = 8k + 2$ and $e(H_k) = 11k + 1$. 
Every cycle of $H_k$ has length in $\{5,7,8,11,14\}$. 
In particular, $H_k$ contains no $(0,3,4 \bmod{6})$-cycle.  

For each positive integer $k$, every face of $H_k$ is either a $5$-face, a $7$-face, or an $8$-face. 
The two faces incident with the edge $x_k y_k$ are both $5$-faces, whereas every edge other than $x_k y_k$ is incident with a $7$-face. 
Moreover, every vertex of $H_k$ lies on a $7$-face. 
\end{remark}

\begin{theorem}[Main Theorem] \label{main}
Let $n \geq 2$, and let $G$ be an $n$-vertex graph containing no $(0,3,4 \bmod 6)$-cycle. 
Then
\[
e(G) \leq \frac{11}{8}n - \frac{7}{4}.
\]
Moreover, equality holds if and only if $n = 8k + 2$ for some nonnegative integer $k$ 
and $G$ is isomorphic to $H_k$. 
\end{theorem}

\begin{corollary}\label{cor_extremal}
For every integer $n\geq 2$, 
\[
  ex(n,\mathcal{C}_{0,3,4 \bmod{6}})
  =
  \left\lfloor
    \frac{11}{8}n-\frac{7}{4}
  \right\rfloor.
\]
\end{corollary}
One of the key steps in the proof of Theorem~\ref{main} is to prove that every graph containing no $(0,3,4 \bmod 6)$-cycle is planar. 
This step is motivated by the result in~\cite{Gyori2026} stating that every graph containing no $(0 \bmod 4)$-cycle is planar. 
Our proof, however, differs from the one given there and takes a more algebraic approach. 

We also compare Theorem~\ref{main} with related extremal problems obtained by forbidding fewer residue classes. 
By a result in~\cite{Chen1994}, for $n \ge 3$, 
\[ ex(n,\mathcal{C}_{0,3 \bmod 6}) = ex(n,\mathcal{C}_{0 \bmod 3}) = 2n-4. \]
Thus the coefficient of $n$ is larger than the coefficient $11/8$ appearing in Theorem~\ref{main}. 

On the other hand, the exact values of 
\[ ex(n,\mathcal{C}_{0,4 \bmod 6})
\quad \text{and} \quad
ex(n,\mathcal{C}_{3,4 \bmod 6})
\]
are not known. 
However, some lower bounds are easy to obtain. 
First, let $G$ consist of a cactus of $\left\lfloor (n-1)/2 \right\rfloor$ triangles, together with isolated vertices if necessary. 
Every cycle of $G$ is a triangle, and hence $G$ contains no $(0,4 \bmod 6)$-cycle. Therefore, 
\[ 
ex(n,\mathcal{C}_{0,4 \bmod 6}) \ge 3\left\lfloor\frac{n-1}{2} \right\rfloor \ge \frac{3}{2}n-3. 
\]
Second, let $G$ consist of $\left\lfloor (n-2)/2 \right\rfloor$ internally disjoint paths of length $3$ sharing the same pair of endvertices, together with isolated vertices if necessary. 
Every cycle of $G$ has length $6$, and hence
\[ 
ex(n,\mathcal{C}_{3,4 \bmod 6}) \ge 3\left\lfloor\frac{n-2}{2}\right\rfloor \ge \frac{3}{2}n-\frac{9}{2}. 
\]
These lower bounds both have coefficient $3/2$, which is strictly larger than $11/8$. 

Thus, for all sufficiently large $n$, forbidding all
$(0,3,4 \bmod 6)$-cycles yields a strictly smaller extremal number than
forbidding any two of these three residue classes.
Therefore, Theorem~\ref{main} describes a genuinely new extremal phenomenon that cannot be obtained from the corresponding results for two residue classes.

The remainder of this paper is organized as follows. 
In Section~2, we introduce the notation and preliminary definitions. 
In Section~3, we establish a structural lemma for certain $8$-angulations, which will be used in the proof of the main theorem. 
In Section~4, we prove that every graph containing no $(0,3,4 \bmod 6)$-cycle is planar. 
In Section~5, we prove Theorem~\ref{main}, including the characterization of the equality case. 
Finally, in Section~6, we prove the exact formula for the extremal number by constructing an extremal graph for every $n\geq2$.

%preliminaries
\section{Preliminaries}
Unless otherwise stated, all graphs in this paper are simple. 
Let $G$ be a graph. 
For a path $P$ or a cycle $C$ in $G$, we denote by $l(P)$ and $l(C)$ the length of $P$ and $C$, respectively. 
If $l(C) = \ell$, then $C$ is called an \emph{$\ell$-cycle}. If $l(C)$ is even, then $C$ is called an \emph{even cycle}.  
If $P$ has endvertices $x$ and $y$, we write $\End(P) = \{x,y\}$. 
For $X, Y \subseteq V(G)$, if $P$ has endvertices $x$ and $y$ with $x \in X$ and $y \in Y$, 
and no internal vertex of $P$ belongs to $X \cup Y$, then $P$ is called an \emph{$(X,Y)$-path}. 

For a path $P$ and vertices $x,y \in V(P)$, let $P[x,y]$ denote the subpath of $P$ with endvertices $x$ and $y$. 
For an oriented cycle $C$ and vertices $x,y\in V(C)$, let $C[x,y]$ denote the $(x,y)$-path on $C$ directed from $x$ to $y$ along the orientation of $C$. 

For two subgraphs $H_1, H_2$ of a graph $G$, 
we define $H_1 \triangle H_2$ to be the spanning subgraph of $H_1 \cup H_2$ with edge set 
$(E(H_1) \cup E(H_2)) \setminus (E(H_1) \cap E(H_2))$. 
In the proof of the Main Theorem, we often use the fact that, 
for two cycles $C_1$ and $C_2$ of $G$, 
the edge set $E(C_1 \triangle C_2)$ can be decomposed into cycles. 

Throughout this paper, by abuse of notation, for a face $F$ of a plane graph, 
we also write $F$ for the graph obtained from the boundary of $F$. 
Let $G$ be a plane graph. 
For a face $F$ of $G$, if $F$ is a $d$-cycle, then $F$ is called a \emph{$d$-face}. 
The graph $G$ is called a \emph{$d$-angulation} if every face of $G$ is a $d$-face. 

%$\theta_4$-graph
\section{$\theta_4$-graph} 
A graph is called a \emph{$\theta_4$-graph} if it consists of several internally disjoint paths of length $4$ sharing the same pair of endvertices. 
See Figure~\ref{fig_theta_4} for an illustration. 

%%%%%%%%%%%%
\begin{figure}[htbp]
\centering
\includegraphics[width = 10cm]{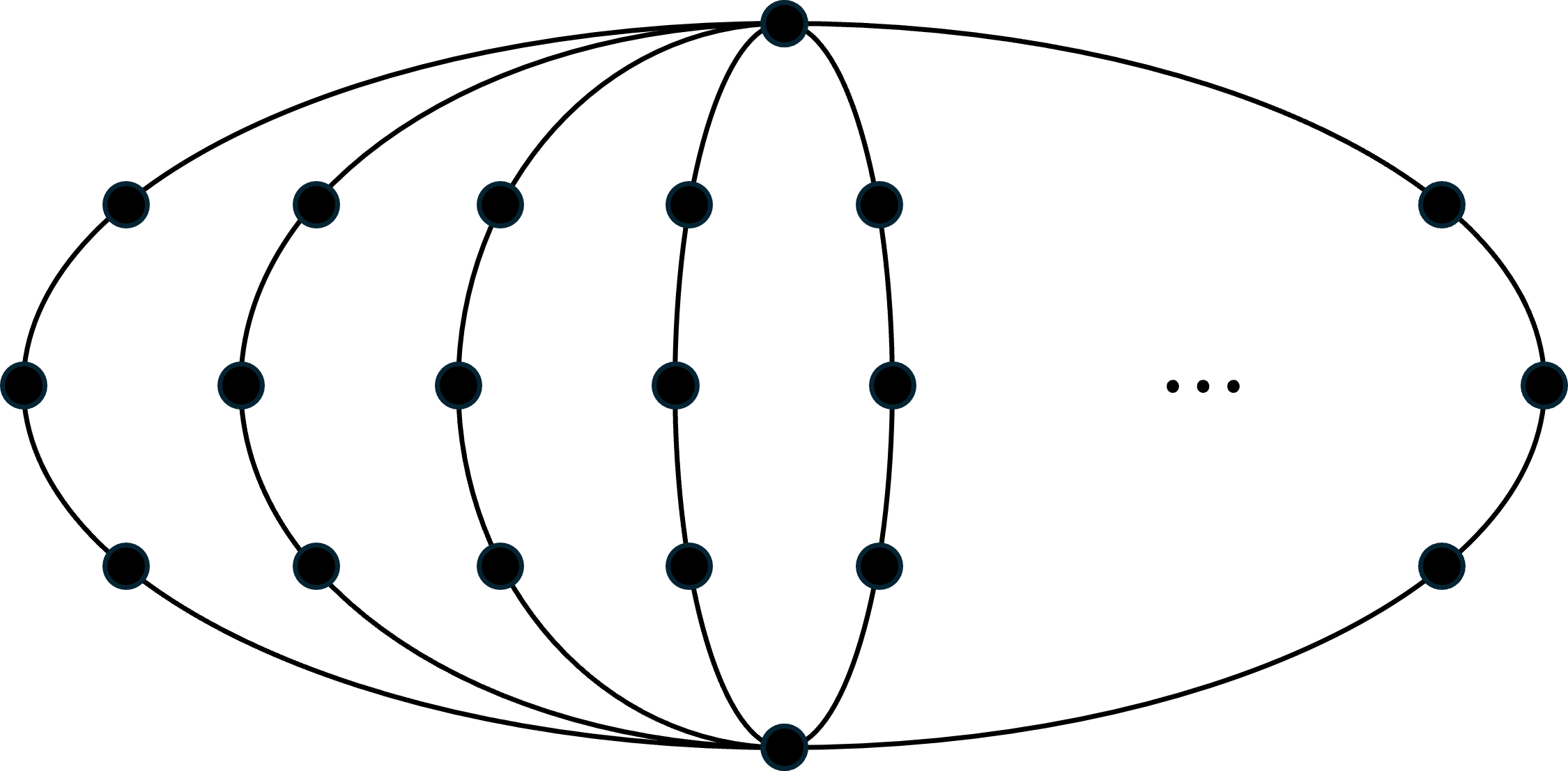}
\caption{$\theta_4$-graph.}
\label{fig_theta_4}
\end{figure}
%%%%%%%%%%%%

In this section, we prove the following lemma, which will be used in the latter part of the proof of the Main Theorem. 

\begin{lemma} \label{lem_theta_4}
Let $H$ be an $8$-angulation whose dual graph $H^*$ is bipartite. 
Suppose that the girth of $H$ is $8$, and that $H$ contains no cycles of length $10$ or $12$. 
Then $H$ is a $\theta_4$-graph.  
\end{lemma}

\begin{proof}
Since every face of $H$ has even length, $H$ is bipartite. 
Let $f_H$ denote the number of faces of $H$.

If $f_H \le 2$, then $H$ is an $8$-cycle, and hence $H$ is a $\theta_4$-graph. 
Thus we may assume that $f_H \ge 3$. 
Take two distinct $8$-faces $C_1$ and $C_2$ of $H$ such that $e(C_1 \cap C_2) \ge 1$. 
Since $f_H \ge 3$, the set $E(C_1 \triangle C_2)$ is not empty. 
Since $H$ is bipartite, the edge set $E(C_1 \triangle C_2)$ can be decomposed into even cycles. 
Moreover, since the girth of $H$ is at least $8$ and $e(C_1 \triangle C_2) \le 14$, 
the graph $C_1 \triangle C_2$ is a cycle. 
Since $H$ contains no cycle of length $10$ or $12$, we have $e(C_1 \triangle C_2) \in \{ 8, 14\}$. 
It follows that $C_1 \cap C_2$ is a path, and
\begin{equation} \label{eq_theta_4_1}
e(C_1 \cap C_2) \in \{ 1, 4\}. 
\end{equation}

By \eqref{eq_theta_4_1}, any two vertices of $H^*$ are joined by $0$, $1$, or $4$ edges. 
Let $H'$ be the graph obtained from $H^*$ by deleting all sets of four parallel edges. 
Since $H^*$ is bipartite, $H'$ is a simple bipartite graph.  
Moreover, the degree of every vertex of $H'$ is divisible by $4$. 

If $H'$ has no edge, then every edge of $H^*$ belongs to a set of four parallel edges, 
and hence $H$ is a $\theta_4$-graph. 
Thus, we may assume that $H'$ has at least one edge. 
Let $D$ be a component of $H'$ having at least one edge. 
Since every vertex of $D$ has degree at least $4$, we have $e(D) \ge 2|V(D)|$ and $|V(D)| \ge 5$. 
On the other hand, since $D$ is a plane bipartite graph with $|V(D)| \ge 5$, we have $e(D) \le 2|V(D)|-4$, a contradiction. 

Therefore $H'$ has no edge, and hence $H$ is a $\theta_4$-graph. 
\end{proof}

%Planarity
\section{Planarity}
In this section, we prove the following proposition. 
\begin{proposition} \label{planarity}
Every graph containing no $(0,3,4 \bmod 6)$-cycle is planar. 
\end{proposition}

To prove the above proposition, we present two lemmas.
\begin{lemma} \label{K_5}
Every subdivision of $K_5$ contains a $(0 \bmod 3)$-cycle. 
In particular, a graph containing no $(0,3,4 \bmod 6)$-cycle cannot contain a subdivision of $K_5$. 
\end{lemma}

\begin{proof}
Throughout this proof, all congruences are taken modulo $3$. 
Let $G$ be a subdivision of $K_5$. Suppose that $G$ contains no $(0 \bmod 3)$-cycles. 
Let $v_1, \ldots, v_5$ be the vertices of $K_5$, and for each edge $v_i v_j$ let $P_{i,j}$ denote the corresponding path in $G$ and set $a_{i,j} = l(P_{i,j})$. 

Suppose that there is no pair $i,j$ with $a_{i,j} \equiv 0$. 
If $a_{i,j} \equiv 1$, we color the edge $v_i v_j$ red, and if $a_{i,j} \equiv 2$, we color the edge $v_i v_j$ blue. 
In this coloring, if there exists a monochromatic triangle, then $G$ contains a $(0 \bmod 3)$-cycle, which is a contradiction.
By Ramsey theory and the symmetry among the vertices $v_1, \ldots, v_5$, we may assume that 
\[
a_{i, i+1} \equiv 1, \quad a_{i, i+2} \equiv 2 \quad (i=1, \ldots, 5)
\]
where the indices are taken modulo $5$.
Therefore,
\[
l(P_{1,2} \cup P_{2,5} \cup P_{5,4} \cup P_{4,1})
= a_{1,2} + a_{2,5} + a_{5,4} + a_{4,1} \equiv 1+2+1+2 \equiv 0,
\]
and hence $G$ contains a $(0 \bmod 3)$-cycle, which is a contradiction. 

Hence, there exists a pair $i,j$ such that $a_{i,j} \equiv 0$. 
Without loss of generality, we may assume that $a_{4,5} \equiv 0$. 
By considering the $3$-cycles $v_i v_4 v_5$ for $i=1,2,3$, we obtain that $a_{i,4} + a_{i,5} \equiv 1 \text{ or } 2$. 
Moreover, by considering the $4$-cycles $v_1 v_4 v_2 v_5$, $v_2 v_4 v_3 v_5$, and $v_3 v_4 v_1 v_5$, it follows that
\[
(a_{1,4}+a_{1,5}, \; a_{2,4}+a_{2,5}, \; a_{3,4}+a_{3,5}) \equiv (1,1,1), (2,2,2).
\]
We may assume that
\begin{equation} \label{k5}
(a_{1,4}+a_{1,5}, \; a_{2,4}+a_{2,5}, \; a_{3,4}+a_{3,5}) \equiv (1,1,1)
\end{equation}
since, if necessary, we can multiply all $a_{i,j}$ by $2$ and reduce to the case congruent to $(1,1,1)$. 

First, consider the situation where $a_{i,j} \equiv 0$ for some $1 \leq i \leq 3$ and $4 \leq j \leq 5$. 
Without loss of generality, we may assume that $a_{1,5} \equiv 0$. 
From $a_{1,4}+a_{1,5} \equiv 1$, it follows that $a_{1,4} \equiv 1$. 
Since $a_{1,5} \equiv 0$, by the same reasoning as before we obtain
\[
(a_{2,1}+a_{2,5}, \; a_{3,1}+a_{3,5}, \; a_{4,1}+a_{4,5}) \equiv (1,1,1), (2,2,2).
\]
Since $a_{4,1} = a_{1,4} \equiv 1$ and $a_{4,5} \equiv 0$, it follows that
\[
(a_{2,1}+a_{2,5}, \; a_{3,1}+a_{3,5}) \equiv (1,1).
\]
Therefore, by considering the cycle $v_1 v_3 v_5 v_2 v_4$, it follows that 
\begin{equation*}
\begin{split}
l(P_{1,3} \cup P_{3,5} \cup P_{5,2} \cup P_{2,4} \cup P_{4,1})
&= (a_{3,1}+a_{3,5}) + (a_{2,4}+a_{2,5}) + a_{1,4} \\
&\equiv 1+1+1 \equiv 0
\end{split}
\end{equation*}
and hence $G$ contains a $(0 \bmod 3)$-cycle, which is a contradiction.

Finally, we consider the case where there is no pair $(i,j)$ with $1 \leq i \leq 3$ and $4 \leq j \leq 5$ such that $a_{i,j} \equiv 0$. 
By equation \eqref{k5}, it follows that 
\[
a_{i,j} \equiv 2 \quad (1 \leq i \leq 3,\; 4 \leq j \leq 5).
\]
By considering the two cycles $v_1 v_2 v_5$ and $v_1 v_2 v_4 v_3 v_5$, we obtain $a_{1,2} \equiv 0$.
By the symmetry of the vertices $v_1, v_2, v_3$, we have $a_{2,3}, \; a_{3,1} \equiv 0$.
Hence, by considering the cycle $v_1 v_2 v_3$, it follows that 
\[
l(P_{1,2} \cup P_{2,3} \cup P_{3,1}) \equiv a_{1,2} + a_{2,3} + a_{3,1} \equiv 0,
\]
and hence $G$ contains a $(0 \bmod 3)$-cycle, which is a contradiction.
\end{proof}

\begin{lemma} \label{K_{3,3}}
Every subdivision of $K_{3,3}$ contains a $(0, 4 \bmod 6)$-cycle. 
In particular, a graph containing no $(0,3,4 \bmod 6)$-cycle cannot contain a subdivision of $K_{3,3}$. 
\end{lemma}

\begin{proof}
Throughout this proof, all congruences are taken modulo $6$. 
Let $G$ be a subdivision of $K_{3,3}$. Suppose that $G$ contains no $(0,4 \bmod 6)$-cycle. 
Let $\{x_1, x_2, x_3\}$ and $\{y_1, y_2, y_3\}$ be the bipartite sets of $K_{3,3}$. 
For each edge $x_i y_j$, let $P_{i,j}$ denote the corresponding path in $G$, and set $l(P_{i,j}) = a_{i,j}$.
For $1 \leq i,j \leq 3$, let 
$C_{i,j} = K_{3,3} - \{x_i, y_j\}$.
These are $4$-cycles of $K_{3,3}$. 
Furthermore, let $S_3$ denote the symmetric group on $\{1,2,3\}$, 
namely the set of all bijections from $\{1,2,3\}$ to itself.  
For $\sigma \in S_3$, define
$C_{\sigma} = K_{3,3} - \{x_1 y_{\sigma(1)}, \; x_2 y_{\sigma(2)}, \; x_3 y_{\sigma(3)}\}$.
These are $6$-cycles of $K_{3,3}$.
For a cycle $C$ of $K_{3,3}$, let $l'(C)$ denote the length of the corresponding cycle in $G$. 
Then the following holds, where $k, \ell$ range over integers from $1$ to $3$:
\begin{itemize}
  \item For $1 \leq i,j \leq 3$,
  \[
  l'(C_{i,j}) = \sum_{\substack{k \neq i \\ \ell \neq j}} a_{k,\ell} \equiv 1,2,3,5 \pmod{6}.
  \]
  \item For $\sigma \in S_3$,
  \[
  l'(C_{\sigma}) = \sum_{\ell \neq \sigma(k)} a_{k,\ell}  \equiv 1,2,3,5 \pmod{6}.
  \]
\end{itemize}
For $1 \leq i,j \leq 3$, let $b_{i,j} = l'(C_{i,j})$. Then, $b_{i,j} \equiv 1,2,3,5$. 
Define the $3 \times 3$ matrix $B$ by $B = (b_{i,j})_{1 \leq i,j \leq 3}$. 
By the definitions of $C_{i,j}$ and $b_{i,j}$, a straightforward calculation shows that the sum of the entries in each row and each column of $B$ is even.  

For $\sigma \in S_3$, define
$T_{\sigma} = \sum_{\ell = \sigma(k)} b_{k,\ell}$.
Then
\[
T_{\sigma} = \sum_{k,\ell} a_{k,\ell} + \sum_{\ell = \sigma(k)} a_{k,\ell}.
\]
On the other hand, if we set
$b_{\mathrm{sum}} = \sum_{k,\ell} b_{k,\ell}$,
then
\[
b_{\mathrm{sum}} = 4 \sum_{k,\ell} a_{k,\ell}.
\]
Moreover, we have
\begin{align*}
l'(C_{\sigma}) 
&= \sum_{k,\ell} a_{k,\ell} - \sum_{\ell = \sigma(k)} a_{k,\ell} \\
&= 2 \sum_{k,\ell} a_{k,\ell} - \Bigl( \sum_{k,\ell} a_{k,\ell} + \sum_{\ell = \sigma(k)} a_{k,\ell} \Bigr) \\
&\equiv -4 \sum_{k,\ell} a_{k,\ell} - T_{\sigma} \\
&= -b_{\mathrm{sum}} - T_{\sigma}.
\end{align*}
Summarizing the above, we obtain the following.
\begin{itemize}
  \item For $1 \leq i,j \leq 3$,
  \begin{equation*}
  b_{i,j} \equiv 1,2,3,5 \pmod{6}.
  \end{equation*}
  \item The sum of the entries in each row and each column of $B$ is even. 
  \item For $\sigma \in S_3$,
  \begin{equation} \label{sigma}
  b_{\mathrm{sum}} + T_{\sigma}  \equiv 1,3,4,5 \pmod{6}.
  \end{equation}
\end{itemize}
By considering the parity of the entries of $B$ together with the  symmetry of $K_{3,3}$, 
it suffices to examine only the following cases, where each asterisk $*$ represents one of the values $1$, $3$, or $5$. 
%matrix
\[
\textbf{Case $1$}: \quad
B = \begin{pmatrix}
2 & 2 & 2 \\
2 & 2 & 2 \\
2 & 2 & 2
\end{pmatrix}, \quad
\textbf{Case $2$}: \quad
B = \begin{pmatrix}
2 & 2 & 2 \\
2 & * & * \\
2 & * & *
\end{pmatrix}, \quad
\textbf{Case $3$}: \quad
B = \begin{pmatrix}
2 & * & * \\
* & 2 & * \\
* & * & 2
\end{pmatrix}
\]
First, we consider Case $1$. Let $\mathrm{id}$ denote the identity element of $S_3$.  
Then we have
$b_{\mathrm{sum}} + T_{\mathrm{id}} = 18 + 6 \equiv 0$,
which contradicts equation~(\ref{sigma}).

Next, we consider Case $2$. In this case, note that 
$b_{2,2}, b_{2,3}, b_{3,2}, b_{3,3}$ are odd, and $b_{\mathrm{sum}}$ is even.  
Combining this observation with equation~(\ref{sigma}), we obtain
\[
b_{\mathrm{sum}} + 2 + b_{2,2} + b_{3,3} 
\equiv b_{\mathrm{sum}} + 2 + b_{2,3} + b_{3,2} 
\equiv 4.
\]
In particular,
\[
b_{2,2} + b_{3,3} \equiv b_{2,3} + b_{3,2}.
\]
Let $d = b_{2,2} + b_{3,3}$.  
Since $d$ is even, we have
\begin{align*}
b_{\mathrm{sum}} + T_{\mathrm{id}} 
&= 12 + 2b_{2,2} + 2b_{3,3} + b_{2,3} + b_{3,2} \\
&\equiv 12 + 3d \equiv 0, 
\end{align*}
which contradicts equation~(\ref{sigma}).

Finally, we consider Case $3$. In this case, note that 
$b_{1,2}, b_{2,1}, b_{2,3}, b_{3,2}, b_{3,1}, b_{1,3}$ are odd, and $b_{\mathrm{sum}}$ is even. 
Combining this observation with equation~(\ref{sigma}), we obtain
\[
b_{\mathrm{sum}} + 2 + b_{1,2} + b_{2,1} 
\equiv b_{\mathrm{sum}} + 2 + b_{2,3} + b_{3,2}
\equiv b_{\mathrm{sum}} + 2 + b_{3,1} + b_{1,3}
\equiv 4.
\]
In particular,
\[
b_{1,2} + b_{2,1} \equiv b_{2,3} + b_{3,2} \equiv b_{3,1} + b_{1,3}.
\]
Let $d = b_{1,2} + b_{2,1}$.  
Since $d$ is even, we have
\begin{align*}
b_{\mathrm{sum}} + T_{\mathrm{id}} 
&= 12 +  b_{1,2} + b_{2,1} + b_{2,3} + b_{3,2} + b_{3,1} + b_{1,3}\\
&\equiv 12 + 3d \equiv 0, 
\end{align*}
which contradicts equation~(\ref{sigma}).
\end{proof}

\begin{proof}[(Proof of Proposition~\ref{planarity})]
Let $G$ be a graph containing no $(0,3,4 \bmod 6)$-cycle.  
By Lemmas~\ref{K_5} and \ref{K_{3,3}}, the graph $G$ contains no subdivision of $K_5$ or $K_{3,3}$.  
Therefore, by Kuratowski's theorem, $G$ is planar. 
\end{proof}

%Proof of Main Theorem
\section{Proof of Main Theorem}
We proceed by induction on $n$.
If $2 \leq n \leq 4$, the graph $G$ contains no cycles, and hence 
\[
e(G) \le n-1 \le \frac{11}{8}n - \frac{7}{4}.
\]
Equality holds if and only if $n=2$ and $G \simeq H_0 (\simeq K_2)$. 
Therefore, it suffices to prove the statement for $n \geq 5$. 

Assume that the statement does not hold for $n$. 
That is, suppose that there exists an $n$-vertex graph $G$ containing no $(0,3,4 \bmod 6)$-cycles and satisfying one of the following conditions:
\begin{itemize}
\item $e(G) > (11/8) n - 7/4$,
\item $e(G) = (11/8) n - 7/4$ and, letting $k = (n-2)/8$, we have $G \not \simeq H_k$.
\end{itemize}
Furthermore, among all graphs satisfying one of the above conditions, we redefine $G$ to be one with the maximum number of edges.

\begin{lemma} \label{cut}
The following properties hold.
\begin{enumerate}
  \item[$(1)$] $G$ is $2$-connected.
  \item[$(2)$] If $\{x, y\}$ is a vertex cut of $G$, then $xy \notin E(G)$.
\end{enumerate}
\end{lemma}

\begin{proof}
$(1)$ 
Suppose that $G$ is disconnected. 
Then we may add an edge between two distinct components of $G$ without creating any new cycle. 
The resulting graph still contains no $(0,3,4 \bmod 6)$-cycle and has more edges than $G$, 
contradicting the choice of $G$ as a graph with the maximum number of edges. 
Therefore, $G$ is connected. 

Suppose that $G$ is not $2$-connected. 
Then there exist proper induced subgraphs $G_1$ and $G_2$ of $G$ such that  
$G = G_1 \cup G_2$, $|V(G_1)| \ge 2$, $|V(G_2)| \ge 2$, and $\lvert V(G_1) \cap V(G_2) \rvert  = 1$.
For $i = 1,2$, set $n_i = |V(G_i)|$.  
Then $n_1 + n_2 = n + 1$ and $2 \le n_1, n_2 \le n-1$.  
By the induction hypothesis, we have
$e(G_i) \leq (11/8) n_i - 7/4$. 
Hence,
\[
e(G) = e(G_1) + e(G_2)
   \leq \left(\frac{11}{8} n_1 - \frac{7}{4}\right)
      + \left(\frac{11}{8} n_2 - \frac{7}{4}\right)
   = \frac{11}{8} n - \frac{17}{8}
   < \frac{11}{8} n - \frac{7}{4},
\]
a contradiction.  
Therefore, $G$ is $2$-connected.

$(2)$ Suppose that $\{x, y\}$ is a vertex cut of $G$ and $xy \in E(G)$. 
Then there exist proper induced subgraphs $G_1$ and $G_2$ of $G$ such that  
$G = G_1 \cup G_2$, $|V(G_1)| \ge 3$, $|V(G_2)| \ge 3$, and $V(G_1) \cap V(G_2) = \{x, y\}$.
For $i = 1,2$, set $n_i = |V(G_i)|$.  
Then $n_1 + n_2 = n + 2$ and $3 \le n_1, n_2 \le n-1$.  
By the induction hypothesis, we have
$e(G_i) \leq (11/8) n_i - 7/4$. 
Hence,
\[
e(G) = e(G_1) + e(G_2) - 1
   \leq \left(\frac{11}{8} n_1 - \frac{7}{4}\right)
      + \left(\frac{11}{8} n_2 - \frac{7}{4}\right) - 1
   = \frac{11}{8} n - \frac{7}{4}.
\]
Therefore, by the assumption on $G$ we have $e(G) = (11/8) n - 7/4$, and hence for each $i=1,2$ we obtain $e(G_i) = (11/8) n_i - 7/4$. 
Setting $k_i = (n_i-2)/8$, we see that $k_i$ is an integer and $G_i \simeq H_{k_i}$.
From $n_1, n_2 \ge 3$, we obtain $k_1, k_2 \ge 1$.

For each $i=1,2$, let $\varphi_i \colon G_i \to H_{k_i}$ be an isomorphism. 
Let $C_1, C_2$ be the two faces of $G_1$ incident with the edge $xy$, and let $C'_1, C'_2$ be the two faces of $G_2$ incident with the edge $xy$. 
For $1 \le s,t \le 2$, $C_s \triangle C'_t$ is a cycle of $G$, and hence
\[
l(C_s) + l(C'_t) - 2 = l(C_s \triangle C'_t) \equiv 1,2,5 \pmod{6}.
\]
In particular, $\left( l(C_s), l(C'_t) \right) \notin \{(5,7), (7,5), (7,7) \}$. 
Therefore, by Remark~\ref{rmk:Hk}, 
\[ \varphi_i(xy) = x_{k_i} y_{k_i}\]
for each $i = 1,2$. 
It follows that $G \simeq H_{k_1 + k_2}$, contradicting the assumption on $G$.
\end{proof}

By Lemma~\ref{cut}, the graph $G$ is $2$-connected. 
Moreover, by Proposition~\ref{planarity}, $G$ is planar. 
Fix an embedding of $G$ into the plane, and regard $G$ as a plane graph.
Since $G$ is $2$-connected and planar, every face of $G$ is a cycle. 

Before considering the $5$-faces of $G$, we establish the following lemma.

\begin{lemma} \label{lem_5f}
If $C_1, C_2$ are two $5$-faces of $G$, then they intersect at exactly one edge. 
\end{lemma}
\begin{proof}
Since $G$ contains no $(0,3,4 \bmod 6)$-cycle, we note that the girth of $G$ is at least $5$. 

Suppose that $C_1, C_2$ are two $5$-faces that do not intersect at exactly one edge.
We distinguish the following four cases:
\begin{itemize}
  \item $C_1 \cap C_2$ is disconnected.
  \item $C_1 \cap C_2$ is a path of length at least $2$.
  \item $C_1 \cap C_2 = \emptyset$.
  \item $C_1 \cap C_2$ is $K_1$.
\end{itemize}

Suppose first that $C_1 \cap C_2$ is disconnected.  
In this case, $E(C_1 \triangle C_2)$ can be decomposed into at least two cycles.  
Since the girth of $G$ is at least $5$, it follows that $E(C_1 \triangle C_2)$ can be decomposed into two $5$-cycles.  
Therefore, $C_1 \cap C_2$ is $2K_1$ (two isolated vertices).
Let $V(C_1 \cap C_2) = \{x, y\}$.  
Then $C_1 \triangle C_2$ contains a cycle of length 
\[
\dist_{C_1}(x,y) + \dist_{C_2}(x,y) \leq 4,
\]
which contradicts the fact that the girth of $G$ is at least $5$.

Suppose second that $C_1 \cap C_2$ is a path of length at least $2$. 
In this case, $C_1 \triangle C_2$ becomes either a $4$-cycle or a $6$-cycle, 
which is a contradiction.

From the above, it suffices to consider the cases where $C_1 \cap C_2 = \emptyset$ or $C_1 \cap C_2 \simeq K_1$, and we prove the following claim.

\begin{claim} \label{claim_5f_1}
The following properties hold (see Figure~\ref{fig:claim_5f_1}).
\begin{enumerate}
\item[$(1)$]
Suppose that $C_1 \cap C_2 = \emptyset$. 
Then there exist edges $v_1 v_2 \in E(C_1)$ and $w_1 w_2 \in E(C_2)$, and vertex-disjoint $(C_1, C_2)$-paths $P_1$ and $P_2$ such that $\End(P_i) = \{v_i, w_i\}$ for $i=1,2$. 
Moreover, there exist vertices $x \in V(P_1)$ and $y \in V(P_2)$ such that $\{x,y\}$ is a vertex cut of $G$. 

\item[$(2)$]
Suppose that $C_1 \cap C_2$ consists of a single vertex. 
Let $C_1 \cap C_2 = \{x\}$. 
Then there exist edges $xv_2 \in E(C_1)$ and $xw_2 \in E(C_2)$, and a $(C_1-x, C_2-x)$-path $P_2$ in $G - x$ such that $\End(P_2) = \{v_2, w_2\}$. 
Moreover, there exists a vertex $y \in V(P_2)$ such that $\{x, y\}$ is a vertex cut of $G$. 
\end{enumerate}
\end{claim}

\begin{proof}
%%%%%%%%%%%%
\begin{figure}[htbp]
\centering
\includegraphics[width = 0.9\textwidth]{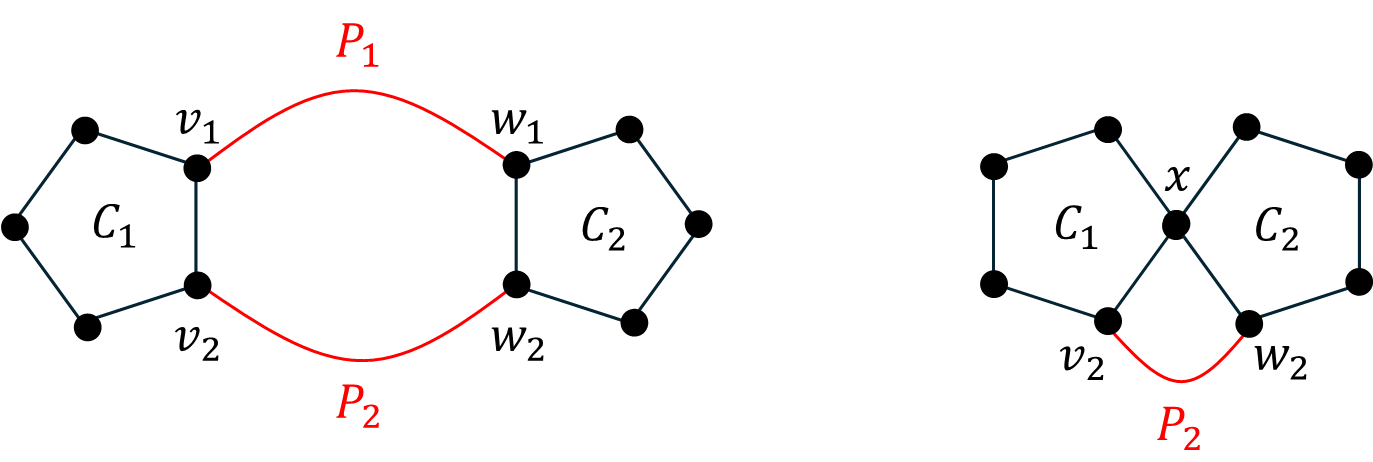}
\caption{Claim~\ref{claim_5f_1} $(1)$ and $(2)$}
\label{fig:claim_5f_1}
\end{figure}
%%%%%%%%%%%%

$(1)$ Suppose that $C_1 \cap C_2 = \emptyset$. 
Since $G$ is $2$-connected, there are two vertex-disjoint $(C_1, C_2)$-paths $P_1$, $P_2$. 
For $i = 1,2$, let $\End(P_i) \cap V(C_1) = \{v_i\}$ and $\End(P_i) \cap V(C_2) = \{w_i\}$, and put $l(P_i) = a_i$.  
Let $d_1 = \dist_{C_1}(v_1, v_2)$ and $d_2 = \dist_{C_2}(w_1, w_2)$.
If $(d_1, d_2) = (2,2)$, then $C_1 \cup C_2 \cup P_1 \cup P_2$ contains cycles of lengths  
$a_1 + a_2 + 4$, $a_1 + a_2 + 5$, and $a_1 + a_2 + 6$, 
one of which is congruent to $0$ or $3 \pmod{6}$, a contradiction.
If $(d_1, d_2) = (1,2)$ or $(2,1)$, then $C_1 \cup C_2 \cup P_1 \cup P_2$ contains cycles of lengths  
$a_1 + a_2 + 3$, $a_1 + a_2 + 4$, $a_1 + a_2 + 6$, and $a_1 + a_2 + 7$, 
one of which is congruent to $3$ or $4 \pmod{6}$, again a contradiction.
Thus, $(d_1,d_2) = (1,1)$. Hence, $v_1v_2 \in E(C_1)$ and $w_1w_2 \in E(C_2)$, which proves the first part.

Now suppose that there exists a $(C_1, C_2)$-path $P_3$ such that 
$P_1, P_2$, and $P_3$ are vertex-disjoint, and let $\End(P_3) \cap C_1 = \{v_3\}$.
By the previous argument, $v_1, v_2$, and $v_3$ must be mutually adjacent on $C_1$, a contradiction.  
Therefore, there exists a vertex cut $\{x, y\}$ in $G$ with $x \in V(P_1)$ and $y \in V(P_2)$, 
which completes the proof. 

$(2)$ Suppose that $C_1 \cap C_2 = \{x\}$. 
Since $G$ is $2$-connected, there is a $(C_1-x, C_2-x)$-path $P_2$ in $G - x$. 
Let $\End(P_2) \cap V(C_1) = \{v_2\}$ and $\End(P_2) \cap V(C_2) = \{w_2\}$, and put $l(P_2) = a_2$.
By considering cycles contained in $C_1 \cup C_2 \cup P_2$, we obtain, as in $(1)$, that $xv_2 \in E(C_1)$ and $xw_2 \in E(C_2)$, which proves the first part. 
Since $C_1 \cup C_2 \cup P_2$ contains cycles of lengths $a_2+2$, $a_2+5$, and $a_2+8$, it follows that $a_2 \equiv 0 \pmod{3}$.

Now suppose that there exists a $(C_1-x, C_2-x)$-path $P_3$  in $G - x$ such that 
$P_2$ and $P_3$ are vertex-disjoint.
Let $\End(P_3) \cap C_1 = \{v_3\}$ and $\End(P_3) \cap C_2 = \{w_3\}$, and put $l(P_3) = a_3$. 
By the previous argument, $xv_3 \in E(C_1)$, $xw_3 \in E(C_2)$ and $a_3 \equiv 0 \pmod{3}$. 
Since the graph $(C_1 \cup C_2 \cup P_2 \cup P_3) - x$ is a cycle of length $a_2 + a_3 + 6 \equiv 0 \pmod{3}$, we obtain a contradiction. 
Therefore, there exists $y \in V(P_2)$ such that $\{x, y\}$ is a vertex cut in $G$, 
which completes the proof. 
\end{proof}

Let $\{x,y\}$ be a vertex cut of $G$ as in Claim~\ref{claim_5f_1}.
Then there exist proper induced subgraphs $G_1$ and $G_2$ of $G$ such that  
$G = G_1 \cup G_2$, $V(G_1) \cap V(G_2) = \{x, y\}$, $V(C_1) \subseteq V(G_1)$ and $V(C_2) \subseteq V(G_2)$.

Claim~\ref{claim_5f_1} yields two paths in $G_1$, denoted by $Q_1,Q_2$, and two paths in $G_2$, denoted by $R_1,R_2$, defined as follows:
\[
Q_1=
\begin{cases}
P_1[x,v_1]\cup v_1v_2\cup P_2[v_2,y]
& (C_1\cap C_2=\emptyset),\\
xv_2\cup P_2[v_2,y]
& (\text{otherwise}), 
\end{cases}
\]
and let $Q_2=Q_1\triangle C_1$. Then $l(Q_2)=l(Q_1)+3$.
Similarly, define
\[
R_1=
\begin{cases}
P_1[x,w_1]\cup w_1w_2\cup P_2[w_2,y]
& (C_1\cap C_2=\emptyset),\\
xw_2\cup P_2[w_2,y]
& (\text{otherwise}), 
\end{cases}
\]
and let $R_2=R_1\triangle C_2$. Then $l(R_2)=l(R_1)+3$.

\begin{claim} \label{claim_5f_2}
The following properties hold.
\begin{enumerate}
\item[$(1)$]
For every $(x,y)$-path $Q$ in $G_1$, we have
$l(Q) \equiv l(Q_1) \pmod{3}$.

\item[$(2)$]
For every $(x,y)$-path $R$ in $G_2$, we have
$l(R) \equiv l(R_1) \pmod{3}$.

\item[$(3)$]
$l(Q_1)+l(R_1) \equiv 2 \pmod 3$.
\end{enumerate}
\end{claim}

\begin{proof}
We first prove $(1)$ and $(3)$. 
Let $Q$ be an arbitrary $(x,y)$-path in $G_1$. 
Since $Q\cup R_1$ and $Q\cup R_2$ are cycles in $G$, 
both $l(Q)+l(R_1)$ and $l(Q)+l(R_2)$ are  congruent to $1$, $2$, or $5 \pmod{6}$. 
Since $l(R_2)=l(R_1)+3$, this is possible only when
\[l(Q)+l(R_1) \equiv 2 \pmod{3}. \] 
In particular, taking $Q=Q_1$, we obtain 
\[ l(Q_1)+l(R_1) \equiv 2 \pmod{3},  \]
which proves $(3)$. 
Comparing these two congruences, we get 
$l(Q) \equiv l(Q_1) \pmod{3}$, and hence $(1)$ follows. 

The proof of $(2)$ is analogous to that of $(1)$. 
Indeed, by considering the cycles $Q_1 \cup R$ and $Q_2 \cup R$ for an arbitrary $(x,y)$-path $R$ in $G_2$, 
and using $l(Q_2) = l(Q_1)+3$, we obtain $l(R) \equiv l(R_1) \pmod{3}$. This proves $(2)$. 
\end{proof}
We return to the proof of Lemma~\ref{lem_5f}.

First, suppose that $l(Q_1)\equiv 1 \pmod{3}$.
By Claim~\ref{claim_5f_2} $(3)$, we have $\left(l(Q_1), l(R_1)\right) \equiv (1,1) \pmod{3}$. 
Since $\{x,y\}$ is a vertex cut of $G$, Lemma~\ref{cut} implies that $xy\notin E(G)$.
By Claim~\ref{claim_5f_2} $(1)$ and $(2)$, every $(x,y)$-path in $G$ has length congruent to $1$ modulo $3$.
Hence, $G+xy$ contains no $(0,3,4 \bmod 6)$-cycles.
This contradicts the maximality of $e(G)$.

Finally, suppose that $l(Q_1)\not \equiv 1 \pmod{3}$.
By Claim~\ref{claim_5f_2} $(3)$, we have $\left(l(Q_1), l(R_1)\right) \equiv (0,2) \text{ or } (2,0) \pmod{3}$.
By symmetry between $G_1$ and $G_2$, 
we may assume that $\left(l(Q_1), l(R_1)\right) \equiv (0,2) \pmod {3}$.

Define a graph $G_1'$ as follows:
\[
V(G_1')=V(G_1)\cup\{z\},
\qquad
E(G_1')=E(G_1)\cup\{xz,yz\}.
\]
Let $G_2'$ be the graph obtained from $G_2+xy$ by contracting the edge $xy$ into a vertex $w$ and suppressing all resulting parallel edges. 

By Claim~\ref{claim_5f_2} $(1)$, every $(x,y)$-path in $G_1$ has length congruent to $0$ modulo $3$.
Hence, every cycle of $G_1'$ containing $z$ has length congruent to $2$ modulo $3$. 
Since every cycle of $G_1'$ not containing $z$ is already a cycle of $G_1$, it follows that $G_1'$ contains no $(0,3,4 \bmod 6)$-cycle. 

Similarly, by Claim~\ref{claim_5f_2} $(2)$, every $(x,y)$-path in $G_2$ has length congruent to $2$ modulo $3$.
Every new cycle of $G'_2$ created by identifying $x$ and $y$ corresponds to an $(x,y)$-path in $G_2$, and hence has length congruent to $2$ modulo $3$. 
Since every other cycle of $G'_2$ corresponds to a cycle of $G_2$, the graph $G'_2$ also contains no $(0,3,4 \bmod 6)$-cycle.
Therefore, both $G_1'$ and $G_2'$ contain no $(0,3,4 \bmod 6)$-cycles. 

For $i=1,2$, let
\[
n_i:=|V(G_i)|,\qquad
n_i':=|V(G_i')|,
\qquad
e_i:=e(G_i),
\qquad
e_i':=e(G_i').
\]

Then $n=n_1+n_2-2, e(G)=e_1+e_2$. 
Moreover, $n_1'=n_1+1, n_2'=n_2-1$ and hence $n_1'+n_2'=n_1+n_2=n+2$. 
Since each of $G_1$ and $G_2$ contains a $5$-cycle, we have
$n_1' \ge 6, n_2' \ge 4$.
In particular, $4\le n_1',n_2'\le n-2$. 
Thus, by the induction hypothesis,
\[
e_i' \le \frac{11}{8} n_i' - \frac{7}{4}
\qquad (i=1,2).
\]
Moreover, since $xy \notin E(G)$, we have $e_1'=e_1+2$ and $e_2' = e_2-|N_{G_2}(x) \cap N_{G_2}(y)|$. 
Since $G$ contains no $4$-cycle, we have $|N_{G_2}(x) \cap N_{G_2}(y)|\le 1$,
and hence $e_2'\ge e_2-1$.
Therefore, $e_1'+e_2' \ge e_1+e_2+1 = e(G)+1$.

Consequently,
\[
\begin{aligned}
e(G)
&\le e_1'+e_2'-1\\
&\le
\left(\frac{11}{8}n_1'-\frac74\right)
+
\left(\frac{11}{8}n_2'-\frac74\right)
-1\\
&=
\frac{11}{8}(n_1'+n_2')-\frac92\\
&=
\frac{11}{8}(n+2)-\frac92\\
&=
\frac{11}{8}n-\frac74.
\end{aligned}
\]
By the assumption on $G$, $e(G) = (11/8) n - 7/4$ 
and therefore equality holds throughout the above inequalities.
In particular, $e_1' = (11/8) n_1' - 7/4$.
By the induction hypothesis, letting $k = (n_1'-2)/8$, we have $k \in \mathbb{Z}_{>0}$ and 
$G_1' \simeq H_k$.
Let $\varphi : G_1' \to H_k$ denote this isomorphism.
Since every cycle of $G_1'$ containing $z$ has length congruent to $2$ modulo $3$, every cycle in $H_k$ through $\varphi(z)$ also has length congruent to $2$ modulo $3$.
This contradicts Remark~\ref{rmk:Hk}.

Therefore, the lemma follows.
\end{proof}

\begin{lemma} \label{5-face}
$G$ has at most two $5$-faces. 
\end{lemma}
\begin{proof}
Suppose that $C_1, C_2, C_3$ are three $5$-faces. 
Let $H = C_1 \triangle C_2 \triangle C_3$.  
By Lemma~\ref{lem_5f}, we have $|E(H)| = 9$.  
Since $E(H)$ can be decomposed into one or more cycles, this contradicts the fact that $G$ contains neither cycles of length at most $4$ nor a $9$-cycle.
\end{proof}

We next consider $7$-faces. 

\begin{lemma} \label{7-face}
The following properties hold. 
\begin{enumerate}
  \item[$(1)$] Let $C_1$ and $C_2$ be distinct $7$-faces of $G$ with $e(C_1 \cap C_2) \ge 1$.  
Then $C_1 \cap C_2$ is a path of length $3$. 
  \item[$(2)$] Every $7$-face of $G$ shares edges with at most one other $7$-face. 
\end{enumerate}
\end{lemma}
\begin{proof}
$(1)$
Let $c$ denote the number of connected components of $C_1 \cap C_2$.
Since $E(C_1\triangle C_2)$ can be decomposed into $c$ cycles and the girth of $G$ is at least $5$, we have
\[
5c \le e(C_1\triangle C_2) \le e(C_1)+e(C_2) = 14.
\]
Hence, $c \le 2$.

If $c=1$, then $C_1\cap C_2$ is a path. 
Let $a$ denote the length of this path. 
Then $a \ge 1$, and $C_1\triangle C_2$ is a cycle of length $14-2a$.
Since $C_1 \triangle C_2$ is not a $(0,3,4 \bmod 6)$-cycle, we must have $a = 3$, and hence the conclusion follows.
Therefore, it suffices to consider the case where $c=2$. 

Let $P_1$ and $P_2$ be the connected components of $C_1\cap C_2$. 
Then each of $P_1$ and $P_2$ is either a path or a single vertex. 
We regard a single vertex as a path of length $0$. 
Fix an orientation of $C_1$. 
For $i=1,2$, let $x_i,y_i$ denote the endvertices of $P_i$ (if $l(P_i)=0$, then let $x_i=y_i$). 
Assume that $x_1,y_1,x_2,y_2$ appear on $C_1$ in this order according to the orientation. 
Choose an orientation of $C_2$ so that $x_1,y_1,x_2,y_2$ also appear in this order on $C_2$. 
See Figure~\ref{fig_7faces} for an illustration. 

%%%%%%%%%%%%
\begin{figure}[htbp]
\centering
\includegraphics[width = 5cm]{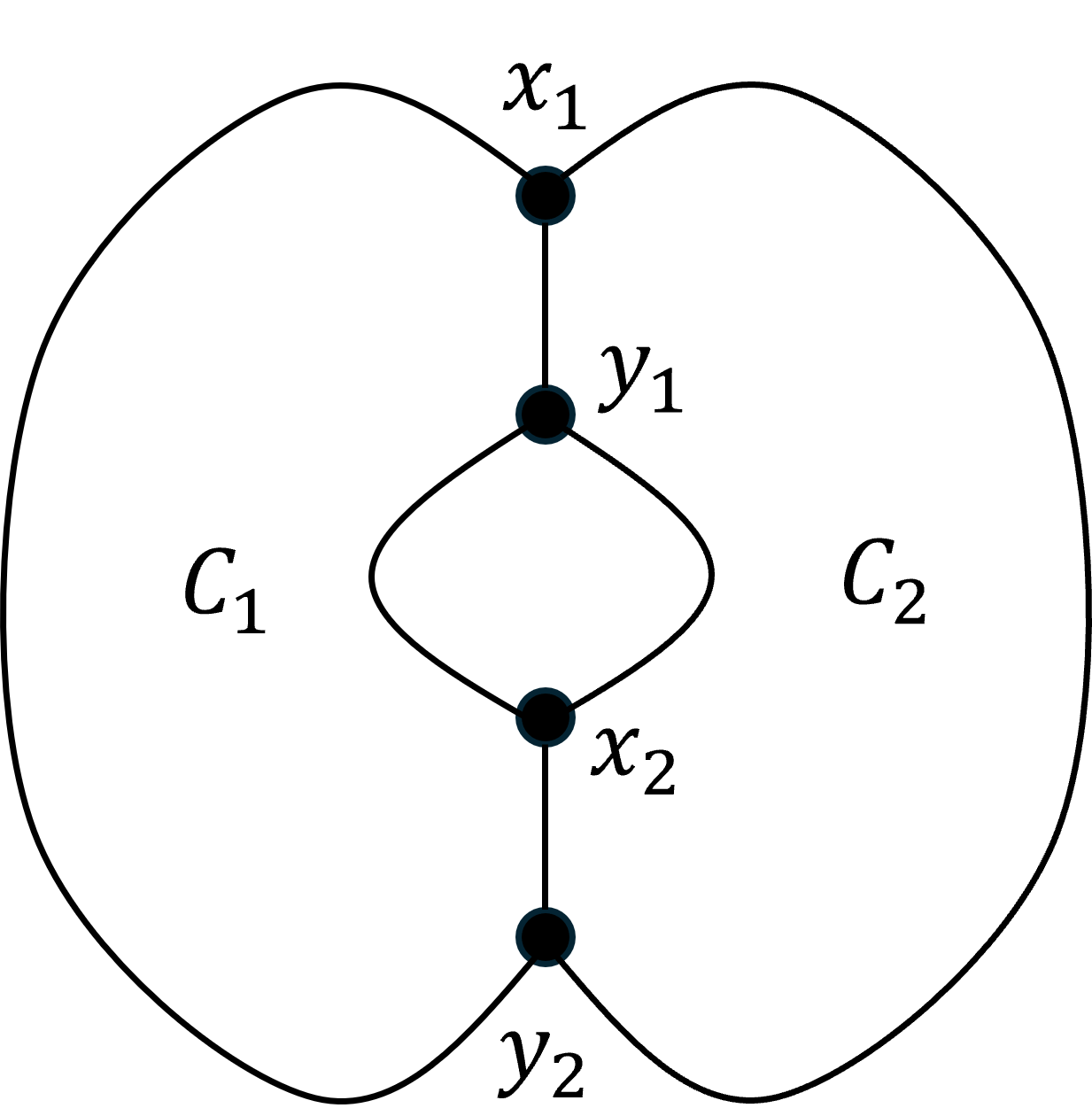}
\caption{The $7$-faces $C_1$ and $C_2$.}
\label{fig_7faces}
\end{figure}
%%%%%%%%%%%%

Define
\[
C_1':=C_1[x_2,y_1]\cup C_2[y_1,x_2],\qquad C_2':=C_1[y_1,x_2]\cup C_2[x_2,y_1].
\]

Then both $C_1'$ and $C_2'$ are cycles in $G$, and $l(C_1')+l(C_2')=l(C_1)+l(C_2)=14$. 
Since $G$ contains no $(0,3,4 \bmod 6)$-cycles, we have $l(C_1')=l(C_2')=7$. 
Therefore, $l(C_1[y_1,x_2])=l(C_2[y_1,x_2])$ and $l(C_1[y_2,x_1])=l(C_2[y_2,x_1])$. 
Hence, the following two cycles:
\[
C_1[y_1,x_2] \cup C_2[y_1,x_2],\qquad C_1[y_2,x_1]\cup C_2[y_2,x_1]
\]
both have even length, and the sum of their lengths is at most $14$. 
This contradicts the fact that $G$ contains no $(0,3,4 \bmod 6)$-cycles. 

$(2)$
Suppose that there exist distinct $7$-faces $C_1,C_2,C_3$ of $G$ such that $e(C_1 \cap C_2) \ge 1$ and $e(C_1 \cap C_3) \ge 1$.
By $(1)$,
\[
e(C_1\cap C_2)=3,
\qquad
e(C_1\cap C_3)=3,
\qquad
e(C_2\cap C_3)\in\{0, 3\}.
\]
Hence,
\[
\begin{aligned}
e(C_1\triangle C_2\triangle C_3)
&=
e(C_1)+e(C_2)+e(C_3) \\
&\quad
-2\bigl(e(C_1\cap C_2)+e(C_2\cap C_3)+e(C_3\cap C_1)\bigr)\\
&\in
\{3, 9\}.
\end{aligned}
\]
However, $E(C_1\triangle C_2\triangle C_3)$ can be decomposed into cycles, contradicting the fact that $G$ contains no $(0,3,4 \bmod 6)$-cycles.
\end{proof}

\begin{definition}[$7$-face block]
Let $B_1$ and $B_2$ be the plane graphs shown in Figure~\ref{fig_B1B2}. 
A $7$-face block of $G$ is defined as follows. 

If a $7$-face $C$ of $G$ shares no edge with any other $7$-face, 
then $C$ is called a $7$-face block of type $B_1$. 

If two distinct $7$-faces $C_1$ and $C_2$ of $G$ share an edge, 
then the subgraph $C_1 \cup C_2$ is called a $7$-face block of type $B_2$. 
\end{definition}

%%%%%%%%%%%%
\begin{figure}[htbp]
\centering
\includegraphics[width = 10cm]{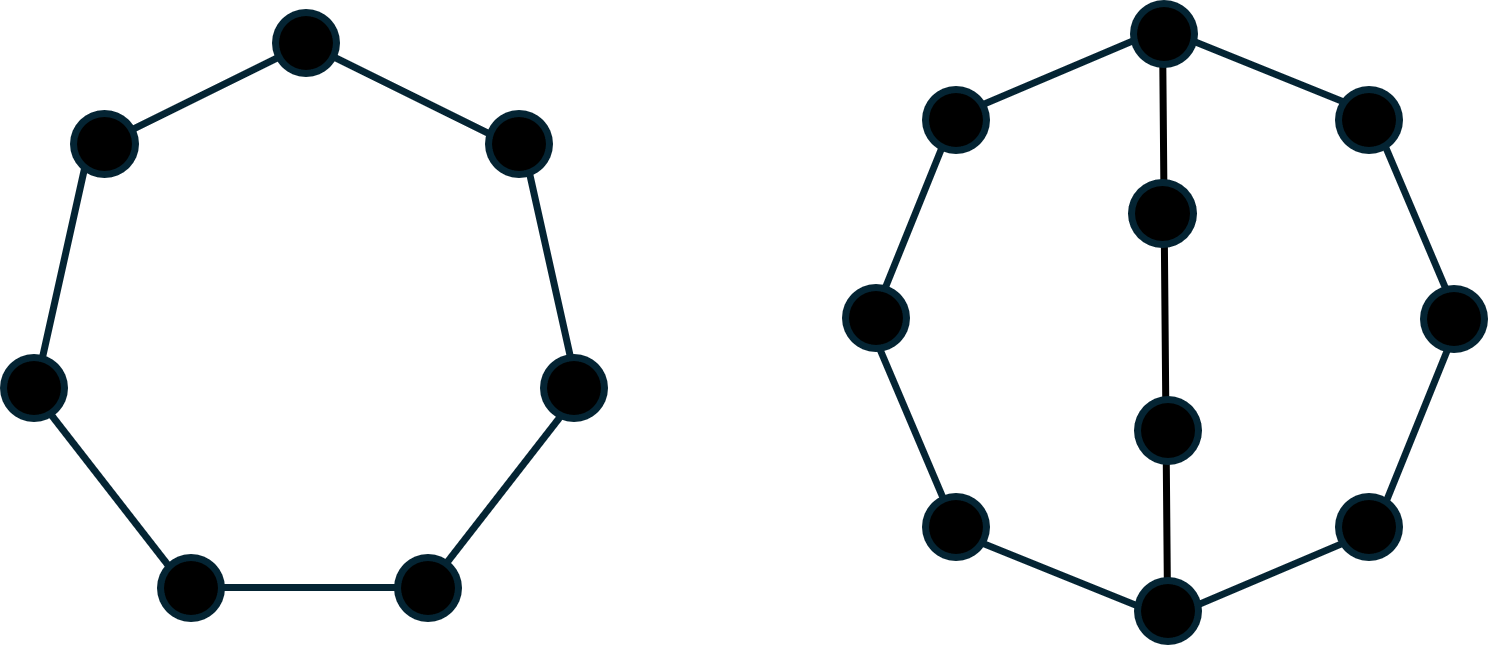}
\caption{The plane graphs $B_1$ (left) and $B_2$ (right).}
\label{fig_B1B2}
\end{figure}
%%%%%%%%%%%%

\begin{remark}
By Lemma~\ref{7-face}, every $7$-face of $G$ is contained in exactly one $7$-face block. 
Moreover, distinct $7$-face blocks are edge-disjoint. 
\end{remark}

Let $f$ denote the number of faces of $G$, and let $e$ denote the number of edges of $G$. 
For each $i\ge 3$, let $f_i$ denote the number of $i$-faces of $G$. 
Since $G$ contains no $(0,3,4 \bmod 6)$-cycles, we have $f_3=f_4=f_6=0$. 
By Lemma~\ref{5-face}, we have $f_5\le 2$. 

Let $b_1,b_2$ denote the numbers of $7$-face blocks of type $B_1$ and $B_2$, respectively. 
Then $f_7=b_1+2b_2$. 
Since $e(B_1)=7$ and $e(B_2)=11$, and distinct $7$-face blocks are edge-disjoint, we have $e\ge 7b_1+11b_2$. 
Therefore, 
\[
f_7=b_1+2b_2
=\frac{2}{11}(7b_1+11b_2)-\frac{3}{11}b_1
\le \frac{2}{11}e.
\]
We next prove
\[
3f_5+f_7
\le \frac{2}{11}e+\frac{64}{11}.
\tag{$*$}
\]
If $f_5 \le 1$, then
\[
3f_5 + f_7 \le 3 + \frac{2}{11}e < \frac{2}{11}e + \frac{64}{11}. 
\]
Thus the inequality in $(*)$ is strict when $f_5 \le 1$. 

Hence we may assume that $f_5=2$. 
Then, by Lemma~\ref{lem_5f}, the two $5$-faces share exactly one edge, and this edge is not contained in any $7$-face block. Hence, $e \ge 7b_1 + 11b_2 + 1$. 
It follows that
\[ f_7 = \frac{2}{11}(7b_1+11b_2)-\frac{3}{11}b_1 \le \frac{2}{11}(e-1) = \frac{2}{11}e-\frac{2}{11}. \]
Therefore, 
\[ 3f_5 + f_7 \le 6+\frac{2}{11}e-\frac{2}{11} = \frac{2}{11}e+\frac{64}{11}. \]
This proves $(*)$. 

Moreover, equality in $(*)$ holds if and only if
\[
f_5=2,\qquad b_1=0,\qquad
f_7=2b_2=\frac{2}{11}(e-1).
\]
Hence, 
\[
\begin{aligned}
2e
&=\sum_{i\ge 3} i f_i \\
&=5f_5+7f_7+\sum_{i\ge 8} i f_i \\
&\ge 5f_5+7f_7+8(f-f_5-f_7) \\
&=8f-3f_5-f_7 \\
&\ge 8f-\frac{2}{11}e-\frac{64}{11}
\qquad \text{(by $(*)$)} \\
&=8(e-n+2)-\frac{2}{11}e-\frac{64}{11}
\qquad \text{(by Euler's formula)} \\
&=\frac{86}{11}e-8n+\frac{112}{11}.
\end{aligned}
\]
Therefore, $e \le (11/8)n - 7/4$. 
By the assumption on $G$, we have $e =  (11/8)n - 7/4$, 
and hence every inequality above must be an equality. 
In particular, $f_i=0$  $(i\ge 9)$, 
and equality holds in $(*)$.

Let $k:=(n-2)/8$. 
Then the above argument yields the following properties:
\begin{itemize}
\item $k\in \mathbb{Z}_{>0}$, $n=8k+2$, $e=11k+1$, and $f=3k+1$.
\item $f_5=2$, $f_7=2k$, $f_8=k-1$, and $f_i=0$ for $i\notin\{5,7,8\}$.
\item $G$ has exactly $k$ $7$-face blocks, and all of them are of type $B_2$.
\item The two $5$-faces of $G$ share an edge.
\end{itemize}

\begin{definition}
Let $H$ be a plane graph, and let $C$ be an even face of $H$. 
If $x,y\in V(C)$ satisfy $\dist_C(x,y) = l(C)/2$, then $\{x, y\}$ is called an \emph{antipodal pair of $C$}. 

Moreover, let $H$ be a plane subgraph of $G$, and let $C$ be a face of $H$ of even length. Let $P$ be a path in $G$. If $\End(P)$ is an antipodal pair of $C$ and every point of the drawing of $P$, except for its endvertices, lies in the interior of $C$, 
then $P$ is called an \emph{antipodal path of $C$}.
\end{definition}

Let $G'$ be the graph obtained from $G$ by deleting the edge shared by the two $5$-faces and the edges shared by the two $7$-faces in each $7$-face block, and then removing isolated vertices.
The above properties imply the following lemma.

\begin{lemma} \label{lem_8gon}
The graph $G'$ is an $8$-angulation with $2k$ $8$-faces, and its plane dual $(G')^*$ is bipartite. 
Moreover, there exist distinct $8$-faces $C_1,C_2,\ldots,C_k,C'$ of $G'$ and paths $Q_1,Q_2,\ldots,Q_k,R$ in $G$ such that:
\begin{itemize}
\item the faces $C_1,C_2,\ldots,C_k$ are pairwise edge-disjoint; 
\item for each $i\in\{1,\ldots,k\}$, $Q_i$ is an antipodal path of $C_i$ with $l(Q_i)=3$; 
\item $R$ is an antipodal path of $C'$ with $l(R)=1$; 
\item 
$G=G'\cup R\cup \bigcup_{i=1}^k Q_i$. 
\end{itemize}
\end{lemma}
\begin{proof}
Let $F_1$ and $F_2$ be the two $5$-faces of $G$, and let $R = F_1 \cap F_2$. 
By Lemma~\ref{lem_5f}, $R$ is an edge. 
Deleting $R$ merges $F_1$ and $F_2$ into the $8$-face
\[ C' = F_1 \triangle F_2. \]
Moreover, $R$ is an antipodal path of $C'$ with $l(R)=1$. 

For each $i \in \{1, \ldots, k \}$, let $F_i^1$ and $F_i^2$ be the two $7$-faces in the $i$th $7$-face block, 
and let $Q_i = F_i^1 \cap F_i^2$. 
By Lemma~\ref{7-face}, $Q_i$ is a path of length $3$. 
Deleting the edges of $Q_i$ and then removing the isolated internal vertices of $Q_i$ merges $F_i^1$ and $F_i^2$ into the $8$-face 
\[ C_i = F_i^1 \triangle F_i^2.  \]
Since $7$-face blocks are pairwise edge-disjoint, $C_1,\ldots,C_k$ are pairwise edge-disjoint. 
Moreover, $Q_i$ is an antipodal path of $C_i$ with $l(Q_i)=3$. 

All other faces of $G$ are already $8$-faces. 
Since the above operations merge $k+1$ pairs of faces and $f = 3k+1$, the number of faces of $G'$ is $(3k+1)-(k+1)=2k$. 
Thus $G'$ is an $8$-angulation with $2k$ $8$-faces. 

It remains to show that $(G')^*$ is bipartite. 
Since $G'$ is an $8$-angulation with $f(G')=2k$ and $C_1, \ldots, C_k$ are pairwise edge-disjoint, 
every edge of $G'$ is incident with exactly one of $C_1, \ldots, C_k$. 
It follows that $(G')^*$ is bipartite, and $\{ C_1, \ldots, C_k\}$ is one of its bipartition classes. 
\end{proof}

Finally, we prove Theorem~\ref{main}. 
By Lemma~\ref{lem_8gon}, $G'$ is an $8$-angulation, and its plane dual $(G')^*$ is bipartite. 
Since every face of $G'$ is even, $G'$ is bipartite. 
Moreover, $G'$ contains no $(0,3,4 \bmod 6)$-cycles.
Hence every cycle contained in $G'$ is a $(2 \bmod 6)$-cycle.
In particular, we may apply Lemma~\ref{lem_theta_4}, and conclude that $G'$ is a $\theta_4$-graph with $2k$ faces.

Combining this with Lemma~\ref{lem_8gon}, there exist vertices $x,y\in V(G')$ and internally disjoint $(x,y)$-paths $P_1,P_2,\dots,P_{2k}$ of length $4$ such that
\[
G'=P_1\cup P_2\cup \cdots \cup P_{2k},
\]
and the faces of $G'$ are
\[
P_1\cup P_2,\,
P_2\cup P_3,\,
\dots,\,
P_{2k-1}\cup P_{2k},\,
P_{2k}\cup P_1.
\]
Furthermore, there exists an antipodal path $R$ of $P_{2k}\cup P_1$, and for each $1 \le i \le k$, there exists an antipodal path $Q_i$ of $P_{2i-1} \cup P_{2i}$ such that
\[
l(R)=1,\qquad l(Q_i)=3 \ \ (1\le i\le k),
\]
and
\[
G=G'\cup R\cup \bigcup_{i=1}^k Q_i.
\]
See Figure~\ref{fig_G} for an illustration. 

%%%%%%%%%%%%
\begin{figure}[htbp]
\centering
\includegraphics[width = 10cm]{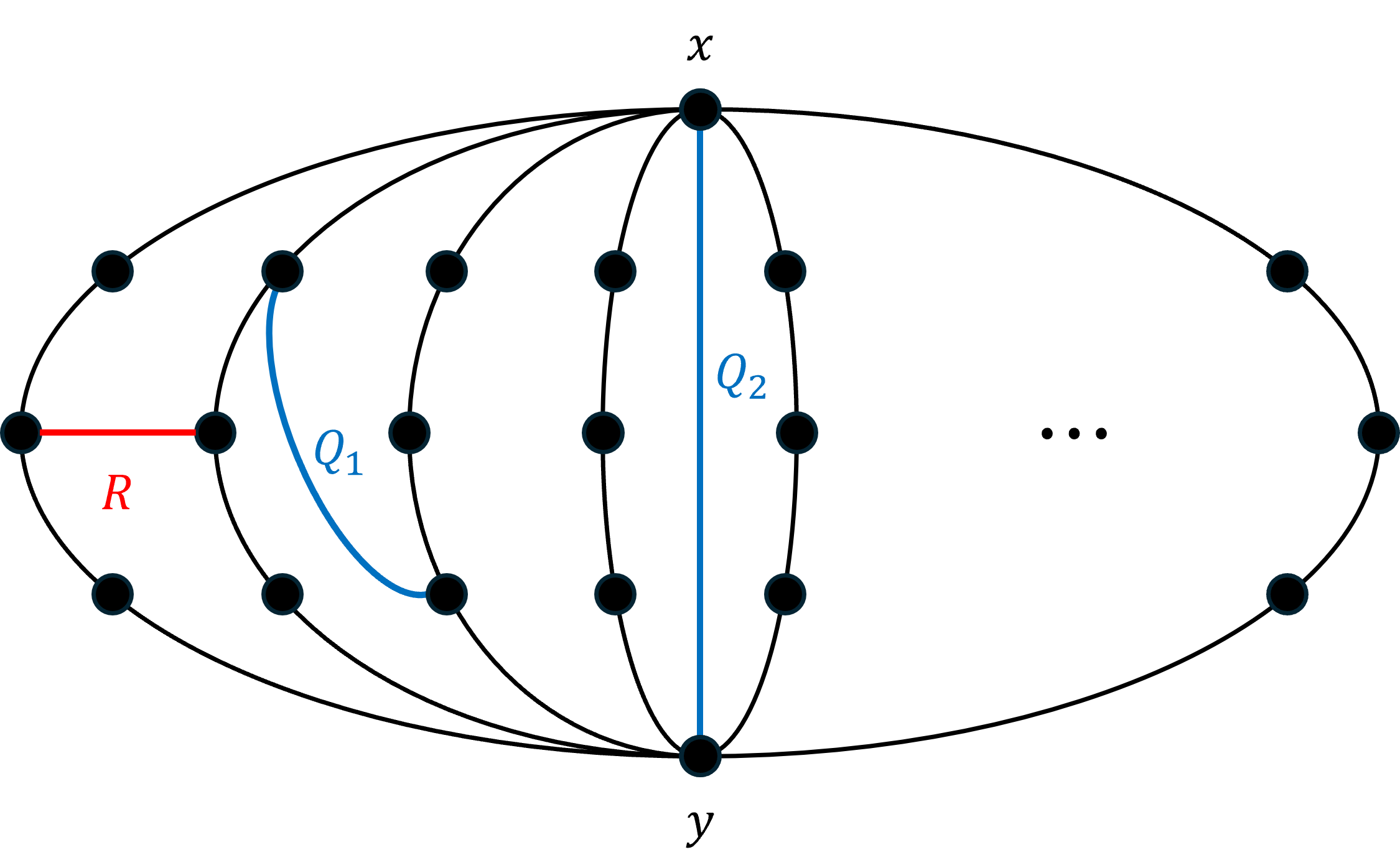}
\caption{The graph $G$.}
\label{fig_G}
\end{figure}
%%%%%%%%%%%%

When $k=1$, the graph $G'$ is isomorphic to an $8$-cycle, and $G$ is obtained from $G'$ by adding antipodal paths of lengths $1$ and $3$ to the two faces of $G'$. 
Up to isomorphism, there are three such graphs, and exactly one of them contains no $(0,3,4 \bmod 6)$-cycle. 
However, that graph is isomorphic to $H_1$, contradicting the assumption on $G$. 
Hence we may assume that $k\ge2$. 

For each $1 \le i \le k$, let $C_{2i-1}',C_{2i}'$ denote the two faces of $G$ separated by $Q_i$. 
Similarly, let $C_1'',C_2''$ denote the two faces of $G$ separated by $R$. 
Moreover, for $1 \le i \le 2k$, let $C_i:=P_{i-1} \cup P_i$ where $P_0:=P_{2k}$. 

Then, for each $1 \le i \le k$,
\[
C_{2i-1}'\cap C_{2i}'=Q_i,\qquad
C_{2i-1}'\triangle C_{2i}'=C_{2i},
\]
and
\[
C_1''\cap C_2''=R,\qquad
C_1''\triangle C_2''=C_1.
\]
Moreover, $C_{2i-1}'$ and $C_{2i}'$ are $7$-faces of $G$, while $C_1''$ and $C_2''$ are $5$-faces of $G$. 
For each $i\in\{1,\dots,2k\}$, let $z_i\in V(P_i)$ satisfy $l(P_i[x,z_i])=2$. 

Fix $1 \le i \le k$.
Since $G$ contains no $9$-cycle, $e(C_{2i-1}' \cap C_{2i-1}) \ne 3$ and $e(C_{2i}' \cap C_{2i-1}) \ne 3$. 
Therefore,
\[
\End(Q_i)\in\bigl\{\{z_{2i-1},z_{2i}\},\{x,y\}\bigr\}.
\]
Let $S := \{1 \le i \le k \mid \End(Q_i) = \{z_{2i-1}, z_{2i} \} \}$, 
and
$T := \{1 \le i \le k \mid \End(Q_i) = \{x,y\} \}$. 
Then $S \cup T=\{1,2,\dots,k\}$. 

Suppose that there exist distinct $i, j \in T$. 
Then $Q_i\cup Q_j$ forms a $6$-cycle, a contradiction. 
Hence $|T| \le 1$. 

Assume that $|T|=1$. 
Then $|S|=k-|T|=k-1 \ge 1$,
and thus there exist $i \in S$ and $j \in T$. 
In this case,
\[
P_{2i}[x,z_{2i}]
\cup Q_i
\cup P_{2i-1}[z_{2i-1},y]
\cup Q_j
\]
forms a $10$-cycle, a contradiction. 
Therefore $T=\emptyset$, and hence $\End(Q_i) = \{z_{2i-1},z_{2i}\}$ 
for every $1 \le i \le k$. 

Finally, we consider $\End(R)$.
If $\End(R) \neq \{ x,y \}$, then it is easy to check that $G$ contains a $(0,3,4 \bmod{6})$-cycle. 
Hence, $\End(R)=\{x,y\}$.  
Therefore, $G$ is isomorphic to $H_k$, contradicting the assumption on $G$.

Hence Theorem~\ref{main} follows. 

\section{The extremal number}
In this section, we prove Corollary~\ref{cor_extremal}. 
\begin{proof}[(Proof of Corollary~\ref{cor_extremal})]
The upper bound follows immediately from Theorem~\ref{main}, since the number of edges is an integer. 

It remains to construct a graph attaining the bound. 
For $n=2$, the graph $H_0 \simeq K_2$ gives the desired construction. 
Thus, assume that $n \ge 3$, and set
\[
  k=\left\lceil\frac{n-2}{8}\right\rceil
  \qquad\text{and}\qquad
  r=8k+2-n.
\]
Then $0\leq r\leq7$. 

Consider one copy $A_1$ of $A$ in $H_k$. 
Label its eight vertices in $V(A_1) \setminus \{ x_k, y_k\}$ as $v_1, v_2, \ldots, v_8$ so that 
\[
  x_k v_3 v_4 v_5 y_k,\qquad
  x_k v_6 v_7 v_8 y_k,\qquad
  v_4 v_1 v_2 v_7
\]
are paths whose edge sets partition $E(A_1)$. 
Let
\[
  G_n=H_k-\{v_1,\ldots,v_r\}.
\]
Since $G_n$ is a subgraph of $H_k$, it contains no $(0,3,4 \bmod 6)$-cycle. 
Moreover, it is easy to check that $|V(G_n)| = n$ and $e(G_n) = \left\lfloor (11/8)n-7/4 \right\rfloor$. 
The result follows. 
\end{proof}
\section{Acknowledgements}
I am deeply grateful to Professor Kenta Ozeki of Yokohama National University for his many valuable suggestions and insightful improvements to the proofs. 
I would like to express my sincere appreciation here. 
The author was supported by the YNU-SPRING Program of Yokohama National University
(JST SPRING, Grant Number JPMJSP2178).


\begin{thebibliography}{99}
\bibitem{Bai2025}
Y.~Bai, A.~Grzesik, B.~Li and M.~Prorok, 
Cycle lengths in graphs of given minimum degree, 
\emph{Journal of Combinatorial Theory, Series B} 180 (2026), 111--150. 

\bibitem{Bai2026}
Y.~Bai, H.~Chu, B.~Li, B.~Park, H.~Ryu, 
On $2$-connected graphs without cycles of length $1$ modulo $3$, 
\emph{arXiv preprint arXiv:2606.02356}, 2026. 

\bibitem{Li2025}
Y.~Bai, B.~Li, Y.~Pan and S.~Zhang, 
On graphs without cycles of length $1$ modulo $3$, 
\emph{arXiv preprint arXiv:2503.03504}, 2025. 

\bibitem{Bollobas1977}
B.~Bollob\'{a}s, 
Cycles modulo $k$, 
\emph{Bulletin of the London Mathematical Society} 9 (1) (1977), 97--98. 

\bibitem{Cai2001}
X.~Cai and W.~E.~Shreve, 
$(2 \bmod 4)$-cycles, 
\emph{Ars Combinatoria} 60 (2001), 97--129. 

\bibitem{Chen1994}
G.~Chen and A.~Saito, 
Graphs with a cycle of length divisible by three, 
\emph{Journal of Combinatorial Theory, Series B} 60 (2) (1994), 277--292. 

\bibitem{Chu2025}
H.~Chu, B.~Park, H.~Ryu, 
On $2$-connected graphs avoiding cycles of length $0 $ modulo $4$, 
\emph{arXiv preprint arXiv:2507.12798}, 2025. 

\bibitem{Erdos1976}
P.~Erd\H{o}s, 
Some recent problems and results in graph theory, combinatorics, and number theory, 
\emph{Proc. Seventh SE Conf. Combinatorics, Graph Theory and Computing, Utilitas Math} (1976), 3--14.   

\bibitem{Gao2022}
J.~Gao, Q.~Huo, C.~H.~Liu and J.~Ma, 
A unified proof of conjectures on cycle lengths in graphs, 
\emph{International Mathematics Research Notices} 10 (2022), 7615--7653. 

\bibitem{Gao2024}
J.~Gao, B.~Li, J.~Ma and T.~Xie, 
On two cycles of consecutive even lengths, 
\emph{Journal of Graph Theory} 106 (2) (2024), 225--238. 

\bibitem{Gyori2026}
E.~Gy\H{o}ri, B.~Li, N.~Salia, C.~Tompkins, K.~Varga and M.~Zhu, 
On graphs without cycles of length $0$ modulo $4$, 
\emph{Journal of Combinatorial Theory, Series B} 176 (2026), 7--29. 

\bibitem{Simonovits1974}
M.~Simonovits, 
Extremal graph problems with symmetrical extremal graphs. Additional chromatic conditions.
\emph{Discrete Mathematics} 7 (3--4) (1974), 349--376. 

\bibitem{Sudakov2017}
B.~Sudakov and J.~Verstra\"{e}te, 
The extremal function for cycles of length $\ell \bmod k$.
%\emph{The Electronic Journal of Combinatorics} 24 (1) (2017), 1--7. 
\emph{The Electronic Journal of Combinatorics} 24 (1) (2017). 

\end{thebibliography}
\end{document}